\pdfoutput=1
\documentclass[11pt]{article}
\usepackage{amsmath,amssymb,amsthm}
\usepackage[margin=2.8cm]{geometry}
\usepackage[colorlinks=true,linkcolor=blue,citecolor=blue]{hyperref}
\newtheorem{theorem}{Theorem}[section]
\newtheorem{lemma}[theorem]{Lemma}
\newtheorem{corollary}[theorem]{Corollary}
\newtheorem{proposition}[theorem]{Proposition}
\theoremstyle{remark}
\newtheorem{remark}[theorem]{Remark}
\newcommand{\R}{\mathbb{R}}
\newcommand{\Z}{\mathbb{Z}}
\newcommand{\N}{\mathbb{N}}

\title{The basic tropical polynomials generate the semifield of\\ $r$-symmetric tropical rational functions\thanks{This work
 was partially supported by MEXT Leading Initiative for Excellent Young
 Researchers Grant Number JPMXS0320200347 and JSPS KAKENHI Grant Number
 JP26K17039.}}
\author{Susumu Kubo\thanks{Faculty of Informatics, Showa Women's University,
 Tokyo, Japan (\texttt{s-kubo@swu.ac.jp}), ORCID 0000-0003-1780-9677.}}
\date{}

\begin{document}
\maketitle

\begin{abstract}
Let the symmetric group $S_n$ act on the space of $n \times r$ real matrices
by permuting rows, so that orbits are multisets of $n$ points in $\R^r$. The
basic $r$-symmetric tropical polynomials form a family of
$\binom{n+r}{r}-1$ nonconstant invariants of degree at most $n$ that
separates orbits and embeds the orbit space bi-Lipschitzly. We prove that this family generates
the semifield of all $r$-symmetric tropical rational functions, answering a
question raised in [Kubo, J.~Pure Appl.\ Algebra 223 (2019) 72--85]. By
Derksen's finite generation criterion, the semiring of invariant tropical
\emph{polynomials} is not finitely generated when $r \ge 2$; at the level
of tropical \emph{rational} functions, Derksen showed that for every
permutation group $G \le S_N$ the invariant semifield is generated in
degree at most $N\, p_1 p_2 \cdots p_{|G|}$, with $p_i$ the $i$th prime
--- for the row action of $S_n$, degree $nr\, p_1 \cdots p_{n!}$. The
present result replaces this by generators of degree at most $n$. The
generating expression is a finite minimum with one term per way of
re-assembling a multiset from its sorted columns, each term carrying a
penalty, built from the basic values, that outweighs the error of a wrong
re-assembly by the bi-Lipschitz inequality. The same penalties describe
the image of the basic coordinate map exactly, as the zero set of a
single tropical rational function, and the construction yields an
expression algorithm. Subfamilies of the basic family containing the
single-column values generate if and only if they separate. Abstracted to
an arbitrary permutation group $G \le S_N$, the same mechanism generates
the invariant semifield in degree at most $\max\{N, \binom N2\}$, a
quadratic bound independent of the group order; combined with a
genericity theorem of Cahill, Iverson, Mixon, and Packer, it also yields
$2N + 1$ invariant tropical polynomials that separate orbits and $3N$
that generate, where at least $N$ are necessary for each task. The
quadratic bound is optimal: for the alternating group $A_N$, every
$A_N$-invariant tropical polynomial of degree less than $\binom N2$ is
$S_N$-invariant, so every family separating $A_N$-orbits contains a
member of degree at least $\binom N2$.
\end{abstract}

\noindent\textbf{Keywords:} tropical polynomial; max-plus algebra; symmetric
functions; generating sets; separating invariants; semifield; piecewise
linear function.\\
\textbf{MSC 2020:} 14T10; 13A50; 12K10; 05E05.

\section{Introduction}\label{sec:intro}

Let the symmetric group $S_n$ act on the space $\R^{n \times r}$ of
$n \times r$ matrices by permuting rows; orbits are multisets of $n$ points
in $\R^r$. The \emph{basic $r$-symmetric tropical polynomials} $b_c$,
indexed by $c = (c_1, \dots, c_r) \in \N^r$ with
$|c| := c_1 + \cdots + c_r \le n$, where $\N = \{0, 1, 2, \dots\}$, were
introduced in \cite{Kubo2019}: they are
invariant tropical (max-plus) polynomials of degree $|c| \le n$, and there
are $\binom{n+r}{r}$ of them ($\binom{n+r}{r}-1$ nonconstant, as $b_0 = 0$). In \cite{K1} we proved that they separate
orbits for all $n$ and $r$, with an explicit reconstruction algorithm, and
that the associated coordinate map is a bi-Lipschitz embedding of the orbit
space.

This paper settles the second question raised in the Discussion of
\cite{Kubo2019}, restated in the concluding remarks of \cite{K1}:

\begin{theorem}[= Theorem~\ref{th:gen}]
For every $n$ and $r$, the basic $r$-symmetric tropical polynomials generate
the semifield of $S_n$-invariant tropical rational functions on
$\R^{n \times r}$.
\end{theorem}

The question sits inside the framework of Derksen \cite{Derksen}, which
charted invariant theory over the tropical semifield for an arbitrary
permutation group $G \le S_N$ acting on $\R^N$ and revealed a contrast
with the classical picture. First, the invariant tropical
\emph{polynomial} semiring is finitely generated if and only if $G$ is
generated by transpositions of coordinates \cite[Theorem~1.1]{Derksen}.
Second, at the level of tropical \emph{rational} functions, finite
generation is restored: \cite[Theorem~1.2]{Derksen} shows that the
invariant tropical polynomials of degree at most $N p_1 p_2 \cdots p_{|G|}$
generate, where $p_i$ is the $i$th prime; the trace construction
underlying that theorem yields finitely many generators --- one trace per
monomial whose exponent vector has all $N$ entries below the product
$p_1 \cdots p_{|G|}$, hence $(p_1 \cdots p_{|G|})^{N}$ of them. For the
row permutation action ($N = nr$, $G = S_n$ of order $n!$) the first
criterion fails when $r \ge 2$ --- a row transposition moves $2r \ge 4$
coordinates --- so no finite set of invariant tropical polynomials
generates the polynomial semiring, while the second gives generators of
degree at most $nr\, p_1 \cdots p_{n!}$ (we refer to the product of the
first $n!$ primes as the primorial factor). Our theorem replaces this by
the family of $\binom{n+r}{r}-1$ nonconstant basic generators of degree at
most $n$. Consistently with the
first contrast, the generating expressions we construct use tropical
division (subtraction) in an essential way.

The proof is short and rests on three pillars: two elementary
consequences of the definitions --- the sorted entries of each column of
a matrix $M \in \R^{n \times r}$ are differences of basic values of $M$,
and the multiset of rows of $M$, written $[M]$, is re-assembled from the
sorted columns by one of finitely many \emph{couplings} (tuples of
permutations matching the columns up) --- and the bi-Lipschitz
inequality for the basic coordinate map, from \cite{K1}. Given an invariant tropical rational function $f$, we express
it as a finite minimum over the couplings. The term for a coupling
$\vec\sigma$ evaluates $f$ on the multiset re-assembled from the sorted
columns according to $\vec\sigma$, plus a penalty. For at least one
coupling the re-assembled multiset is $[M]$ itself; the penalty then
vanishes, and that term equals $f(M)$. For every other coupling the penalty --- a
\emph{recoupling discriminant}, the maximal deviation between the actual
basic values of $M$ and those of the re-assembled multiset --- is, by the
bi-Lipschitz inequality, at least a fixed constant times the distance from
$[M]$ to the re-assembled multiset; a sufficiently large integer multiple
of the penalty therefore outweighs the amount by which $f$ can decrease
over that distance --- at most its Lipschitz constant times the distance
--- and the bi-Lipschitz constant calibrates this integer. All
operations involved are $\max$, $\min$, and integer combinations, hence
tropical rational.
Section~\ref{sec:algorithm} turns the proof into an expression algorithm: given an
invariant tropical rational expression, it outputs an expression in the
basic values; the required integer multiplier is obtained either from an
explicit lower bound on the bi-Lipschitz constant or, in its absence, by
trying the multipliers $1, 2, 4, \dots$ and testing the resulting identity
(the test is decidable). The mechanism is not specific to row permutations:
Section~\ref{sec:general} isolates the two axioms it uses (a separating
bi-Lipschitz family, and finitely many tropical rational maps that jointly
recover a representative of each orbit from the family's values) and
derives, for every permutation group $G \le S_N$, generators of the
invariant semifield of degree at most $\max\{N, \binom N2\}$ --- a bound
quadratic in $N$ and independent of the group order;
Remark~\ref{rem:quadraticvsprimorial} compares it with the primorial bound. A genericity theorem of
Cahill, Iverson, Mixon, and Packer \cite{CIMP}, with templates taken in
$\N^N$, yields
$2N + 1$ invariant tropical polynomials that separate orbits; feeding
these into the axioms gives $3N$ generators of the invariant semifield,
and no separating or generating family has fewer than $N$ members
(Corollaries \ref{cor:count} and~\ref{cor:gencount}); here the degrees are not controlled. A
rigidity theorem (Theorem~\ref{th:alternating}) shows that the quadratic degree
bound is optimal: for the alternating group $A_N$, every $A_N$-invariant
tropical polynomial of degree
less than $\binom N2$ is $S_N$-invariant, so every separating --- in
particular every generating --- family for $A_N$ contains a member of
degree at least $\binom N2$ (Corollary~\ref{cor:anlower}).

As a quantitative counterpart one may ask how small a separating --- in
particular, a generating --- subfamily of the basic family can be. A
subfamily version of the main theorem (Corollary~\ref{cor:subfamily}) shows that for
subfamilies containing the single-column values, generating and separating
are equivalent. The determination of the
minimal separating subfamilies is a distinct combinatorial problem, which
we do not pursue here.

\section{Preliminaries}\label{sec:prelim}

\subsection{Basic values, quotient metric, bi-Lipschitz embedding}

Rows of $M \in \R^{n\times r}$ are denoted $m_i = (m_{i1}, \dots, m_{ir})$;
$[M]$ is the orbit (multiset) of $M$ under row permutations. Following
\cite{K1}, for $c \in \N^r$ with
$0 < |c| := c_1 + \cdots + c_r \le n$ the \emph{basic value} is
\begin{equation}\label{eq:bvalue}
b_c(M) \;=\; \max\Bigl\{ \sum_{\alpha=1}^{r} \sum_{i \in I_\alpha} m_{i\alpha}
\;:\; I_1, \dots, I_r \subseteq \{1, \dots, n\} \text{ pairwise disjoint},\
\#I_\alpha = c_\alpha \Bigr\},
\end{equation}
and $b_0(M) = 0$. Each $b_c$ is an invariant tropical polynomial of degree
$|c|$; it is the tropicalization of the classical elementary multisymmetric
function of multidegree $c$ --- the same selections
$I_1, \dots, I_r$, with the outer sum replaced by $\max$ and the products by
sums. The classical functions are vector invariants of $S_n$ in the sense
of Weyl \cite[Chapter~II]{Weyl}. Write $\Phi = (b_c)_{0 < |c| \le n} : \R^{n \times r} \to \R^K$,
$K = \binom{n+r}{r} - 1$, and let
$d([M], [\bar M]) = \min_{\pi \in S_n} \|M - \pi \cdot \bar M\|_2$ be the
quotient metric. By \cite[Corollary~3.8]{K1}, $\Phi$ is a bi-Lipschitz
embedding:
\begin{equation}\label{eq:bilip}
\|\Phi(M) - \Phi(\bar M)\|_2 \;\ge\; c_{n,r}\; d([M], [\bar M])
\qquad \text{for all } M, \bar M ,
\end{equation}
for some constant $c_{n,r} > 0$. (The reverse inequality is not used.) We also use
$d_\infty([M],[\bar M]) := \min_\pi \max_{i,\alpha} |m_{i\alpha} -
\bar m_{\pi(i)\alpha}| \le d([M],[\bar M])$.

\subsection{Tropical rational functions}

Call a function $\ell : \R^m \to \R$ of the form
$\ell(z) = \gamma + \sum_{i=1}^m a_i z_i$ with $\gamma \in \R$ and
$a \in \Z^m$ an \emph{affine function with integer slopes}. A
\emph{tropical polynomial} in $z_1, \dots, z_m$ is a finite maximum
$\max_j \ell_j$ of such functions with $a_j \in \Z_{\ge 0}^m$; its
\emph{degree} is the smallest value, over such representations, of the
largest $\|a_j\|_1$. A
\emph{tropical rational function} is a difference of two tropical
polynomials. Throughout, tropical rational functions are identified when
they agree as functions on $\R^m$. The set of these functions is a
semifield under the operations $\oplus = \max$, $\odot = +$,
$\oslash = -$; ``generating the semifield'' always refers to these
operations, so it is a statement about functions, not about
formal expressions. (The separating and bi-Lipschitz statements of
\cite{Derksen} are likewise statements about functions.)

We use throughout the function-class characterization: \emph{the tropical
rational functions in $z_1, \dots, z_m$ are exactly the continuous piecewise
linear functions $\R^m \to \R$ with finitely many pieces, each an affine
function with integer slopes}. One inclusion is elementary: a
tropical polynomial is continuous and piecewise linear with finitely many
pieces, each with slopes in $\Z_{\ge 0}^m$, and a difference $P - Q$ of two
tropical polynomials is again continuous and piecewise linear; its pieces
live on the common refinement of the pieces of $P$ and of $Q$ (again
finitely many), and on each piece its slope vector is a difference of two
vectors in $\Z_{\ge 0}^m$, hence lies in $\Z^m$. For the converse, let $f$
be continuous piecewise linear with affine pieces
$\ell_1, \dots, \ell_q$ having integer slopes. By Ovchinnikov's max--min
representation \cite{Ovchinnikov} --- stated for continuous piecewise
linear functions on a closed convex domain, here $\R^m$ itself ---
$f = \min_j \max_{i \in A_j} \ell_i$ for
some subsets $A_j \subseteq \{1, \dots, q\}$. Each $\ell_i$ is tropical
rational: write it as
$\bigl( \gamma + \sum_{a_k > 0} a_k z_k \bigr) - \sum_{a_k < 0} (-a_k) z_k$, a
difference of two tropical monomials, the negative slopes entering through
the subtraction. Maxima stay in the class by
$\max(P_1 - Q_1,\, P_2 - Q_2) = \max(P_1 + Q_2,\, P_2 + Q_1) - (Q_1 + Q_2)$,
and minima by $\min(A,B) = A + B - \max(A,B)$; hence
$f = \min_j \max_{i \in A_j} \ell_i$ is tropical rational.

The characterization shows that the class is closed under composition: if
$g$ and $h_1, \dots, h_m$ are tropical rational, then
$g(h_1, \dots, h_m)$ is continuous, and on each cell of the common
refinement of the (finitely many) pieces of $h_1, \dots, h_m$ on which $g$
is evaluated within a single piece of $g$, it is affine with slope vector
an integer combination of integer slope vectors; there are finitely many
such cells, so $g(h_1, \dots, h_m)$ is again in the class. Call a map
$T = (h_1, \dots, h_m) : \R^p \to \R^m$ \emph{componentwise tropical
rational} if each component $h_i$ is; the closure just proved says
precisely that $g \circ T$ is tropical rational for every tropical rational
$g$ on $\R^m$. For the
polynomial map $\Phi$, composing is just substituting the tropical
polynomial expressions $b_c$ for the variables of $g$, and the
sub-semifield of invariants generated by the basic values is exactly
\[
\{\, g \circ \Phi \;:\; g \text{ tropical rational on } \R^K \,\}.
\]

\section{The generation theorem}\label{sec:gen}

\subsection{Four lemmas}\label{sec:lemmas}

\begin{lemma}[sorted columns]\label{lem:marg}
For every $\alpha \in \{1, \dots, r\}$ and $i \in \{1, \dots, n\}$, the
$i$th largest entry of column $\alpha$ of $M$ equals
$b_{i e_\alpha}(M) - b_{(i-1)e_\alpha}(M)$ (with $b_0 := 0$ and
$e_\alpha \in \N^r$ the $\alpha$th standard basis vector). Consequently
the map $s : \R^K \to \R^{n \times r}$,
$s(w)_{i\alpha} := w_{i e_\alpha} - w_{(i-1) e_\alpha}$ --- componentwise
tropical rational, being a difference of coordinates --- carries $\Phi(M)$
to the matrix of sorted columns of $M$.
\end{lemma}

\begin{proof}
By \eqref{eq:bvalue}, $b_{i e_\alpha}$ is the sum of the $i$ largest entries
of column $\alpha$.
\end{proof}

A tuple $\vec\sigma = (\sigma_2, \dots, \sigma_r) \in S_n^{\,r-1}$ is
called a \emph{coupling}.

\begin{lemma}[re-assembly]\label{lem:sheet}
For a matrix $U \in \R^{n\times r}$ with columns $U^{(1)}, \dots, U^{(r)}$
and a coupling $\vec\sigma$, let
$R_{\vec\sigma}(U)$ be the matrix whose $i$th row is
\[
\bigl( U_{i}^{(1)},\, U_{\sigma_2(i)}^{(2)},\, \dots,\, U_{\sigma_r(i)}^{(r)} \bigr) .
\]
Then for every $M$ there exists $\vec\sigma \in S_n^{\,r-1}$ with
$[M] = [R_{\vec\sigma}(s(\Phi(M)))]$.
\end{lemma}

\begin{proof}
Set $U := s(\Phi(M))$, the matrix of sorted columns of $M$
(Lemma~\ref{lem:marg}), with columns $U^{(1)}, \dots, U^{(r)}$. Index the
rows of $M$ as $q_1, \dots, q_n$ in nonincreasing column-$1$ order, so that
$(q_i)_1 = U^{(1)}_i$ by the choice of indexing. For each
$\alpha \in \{2, \dots, r\}$, the multiset of the entries
$(q_1)_\alpha, \dots, (q_n)_\alpha$ equals the multiset of the entries of
$U^{(\alpha)}$; hence there is a bijection
$\sigma_\alpha : \{1, \dots, n\} \to \{1, \dots, n\}$ with
$(q_i)_\alpha = U^{(\alpha)}_{\sigma_\alpha(i)}$ for all $i$. Then the
$i$th row of $R_{\vec\sigma}(U)$ equals
$q_i$, whence $[R_{\vec\sigma}(U)] = [R_{\vec\sigma}(s(\Phi(M)))] = [M]$.
\end{proof}

\begin{lemma}[recoupling discriminants]\label{lem:disc}
For $\vec\sigma \in S_n^{\,r-1}$ define $D_{\vec\sigma} : \R^K \to \R$,
\[
D_{\vec\sigma}(w) \;:=\; \max_{0<|c|\le n} \bigl|\, w_c \;-\; b_c\bigl( R_{\vec\sigma}(s(w)) \bigr) \,\bigr| .
\]
Then $D_{\vec\sigma}$ is tropical rational,
\[
D_{\vec\sigma}(\Phi(M)) \;\ge\; c'_{n,r}\; d_\infty\bigl( [M],\, [R_{\vec\sigma}(s(\Phi(M)))] \bigr),
\qquad c'_{n,r} := \frac{c_{n,r}}{\sqrt K} \;>\; 0 ,
\]
and moreover $D_{\vec\sigma}(\Phi(M)) = 0$ whenever
$[R_{\vec\sigma}(s(\Phi(M)))] = [M]$.
\end{lemma}

\begin{proof}
Tropical rationality: the entries of $R_{\vec\sigma}(s(w))$ are single
components of $s(w)$, each tropical rational in $w$
(Lemma~\ref{lem:marg}); each $b_c(R_{\vec\sigma}(s(w)))$ is a tropical
polynomial in these entries, hence tropical rational in $w$ by the
composition closure of Section~\ref{sec:prelim}; and
$|a - b| = \max(a-b,\, b-a)$. For the remaining claims, set
$w = \Phi(M)$ and $\widehat M := R_{\vec\sigma}(s(w))$; then
$D_{\vec\sigma}(\Phi(M)) = \|\Phi(M) - \Phi(\widehat M)\|_\infty$. The vanishing claim
follows since $[\widehat M] = [M]$ implies $\Phi(\widehat M) = \Phi(M)$. For
the inequality we combine three estimates:
\[
\|\Phi(M) - \Phi(\widehat M)\|_\infty \;\ge\; \frac{\|\Phi(M)-\Phi(\widehat M)\|_2}{\sqrt K}
\;\ge\; \frac{c_{n,r}}{\sqrt K}\, d([M],[\widehat M])
\;\ge\; c'_{n,r}\, d_\infty([M],[\widehat M]) .
\]
The first is the norm comparison in $\R^K$, and the second is the
bi-Lipschitz inequality \eqref{eq:bilip}. The third is the comparison
$d \ge d_\infty$ stated in Section~\ref{sec:prelim}.
\end{proof}

\begin{lemma}[Lipschitz]\label{lem:lip}
Let a permutation group $G \le S_N$ act on $\R^N$ by permuting
coordinates. Every
tropical rational $f$ on $\R^N$ is $\ell^\infty$-Lipschitz with some
constant $L_f$; if $f$ is $G$-invariant then
$|f(x) - f(y)| \le L_f\, d_\infty([x],[y])$, where
$d_\infty([x],[y]) := \min_{g \in G} \|x - g \cdot y\|_\infty$ as in
Section~\ref{sec:prelim}.
\end{lemma}

\begin{proof}
Write $f = P - Q$. Each linearity region of the tropical polynomial $P$
--- the set where one of its affine terms is maximal --- is an
intersection of half-spaces, hence a closed convex polyhedron, and
likewise for $Q$; the nonempty intersections of a region of $P$ with a
region of $Q$ (their common refinement) are therefore finitely many
closed convex cells covering $\R^N$, and $f$ is affine on
each, with slope $a_p \in \Z^N$ on the cell $p$. Put
$L_f := \max_p \|a_p\|_1$. Within one cell,
$|f(x) - f(y)| = |\langle a_p,\, x - y \rangle|
\le L_f\, \|x - y\|_\infty$. For arbitrary $x, y$, each cell, being
convex, meets the segment from $x$ to $y$ in a closed subsegment
(possibly empty), so the segment subdivides at its finitely many cell
crossings into $x = x_0, x_1, \dots, x_\mu = y$ with consecutive points
in a common cell; by continuity and the one-cell bound,
\[
|f(x) - f(y)| \;\le\; \sum_{k} |f(x_k) - f(x_{k+1})|
\;\le\; L_f \sum_{k} \|x_k - x_{k+1}\|_\infty
\;=\; L_f\, \|x - y\|_\infty ,
\]
the last equality because the $x_k$ lie on one segment in order. For the
invariant statement, choose $g \in G$ attaining $d_\infty([x],[y])$; then
$f(y) = f(g \cdot y)$ and
$|f(x) - f(g \cdot y)| \le L_f\, \|x - g \cdot y\|_\infty
= L_f\, d_\infty([x],[y])$.
\end{proof}

\subsection{The theorem}

\begin{theorem}\label{th:gen}
For every $n$ and $r$, the basic $r$-symmetric tropical polynomials generate
the semifield of $S_n$-invariant tropical rational functions on
$\R^{n\times r}$: for every invariant tropical rational $f$ there is a
tropical rational $g$ on $\R^K$ with $f = g \circ \Phi$. One may take
\begin{equation}\label{eq:formula}
g(w) \;=\; \min_{\vec\sigma \in S_n^{\,r-1}} \Bigl[\, f\bigl( R_{\vec\sigma}(s(w)) \bigr) \;+\; \kappa_f\, D_{\vec\sigma}(w) \,\Bigr],
\qquad w \in \R^K,
\end{equation}
for every sufficiently large integer $\kappa_f$.
\end{theorem}

\begin{proof}
Each term of \eqref{eq:formula} is tropical rational: $f$ composed with the
componentwise tropical rational map $w \mapsto R_{\vec\sigma}(s(w))$, plus an
integer multiple of $D_{\vec\sigma}$ (Lemma~\ref{lem:disc}); the finite
minimum preserves the class. Fix $M$ and write $w = \Phi(M)$,
$\widehat M_{\vec\sigma} := R_{\vec\sigma}(s(w))$.

\emph{$g(w) \le f(M)$}: by Lemma~\ref{lem:sheet} some $\vec\tau$ has
$[\widehat M_{\vec\tau}] = [M]$; its term is
$f(\widehat M_{\vec\tau}) + \kappa_f \cdot 0 = f(M)$, the vanishing clause of
Lemma~\ref{lem:disc} giving $D_{\vec\tau}(w) = 0$.

\emph{$g(w) \ge f(M)$ once $\kappa_f \ge L_f / c'_{n,r}$}: for every
$\vec\sigma$,
\[
f(\widehat M_{\vec\sigma}) + \kappa_f D_{\vec\sigma}(w)
\;\ge\; f(M) - L_f\, d_\infty([M],[\widehat M_{\vec\sigma}]) + \kappa_f\, c'_{n,r}\, d_\infty([M],[\widehat M_{\vec\sigma}])
\;\ge\; f(M),
\]
the first inequality by the invariant Lipschitz bound of
Lemma~\ref{lem:lip} and the discriminant lower bound of
Lemma~\ref{lem:disc}, the second because $\kappa_f\, c'_{n,r} \ge L_f$.
\end{proof}

\begin{corollary}\label{cor:degree}
The semifield of $S_n$-invariant tropical rational functions on
$\R^{n \times r}$ is generated by $\binom{n+r}{r} - 1$ elements of degree at
most $n$.
\end{corollary}

\begin{corollary}[image description]\label{cor:image}
The image of $\Phi$ is the zero set of the single nonnegative tropical
rational function $\min_{\vec\sigma} D_{\vec\sigma}$:
\[
\Phi(\R^{n \times r}) \;=\; \bigl\{\, w \in \R^K \;:\;
\min_{\vec\sigma \in S_n^{\,r-1}} D_{\vec\sigma}(w) \,=\, 0 \,\bigr\} ,
\]
a finite union of polyhedra.
\end{corollary}

\begin{proof}
If $w = \Phi(M)$ then $D_{\vec\tau}(w) = 0$ for any $\vec\tau$ provided by
Lemma~\ref{lem:sheet}, by Lemma~\ref{lem:disc}. Conversely, if the
minimum vanishes at $w$, then $D_{\vec\sigma}(w) = 0$ for some
$\vec\sigma$, and $w = \Phi\bigl( R_{\vec\sigma}(s(w)) \bigr)$. For the
last claim: as in the proof of Lemma~\ref{lem:lip}, $\R^K$ is covered
by finitely many closed convex cells $p$ on each of which
$\min_{\vec\sigma} D_{\vec\sigma}$ agrees with an affine function
$\ell_p$, so its zero set is $\bigcup_p \bigl( p \cap \{\ell_p = 0\}
\bigr)$, a finite union of polyhedra.
\end{proof}

\begin{remark}[relations]\label{rem:sft}
Classically, a presentation of an invariant ring has two halves: the
generators, and the \emph{relations} among them --- the identities that
the generators satisfy, studied for multisymmetric functions since Junker
\cite{Rydh,Vaccarino}. Knowing the relations amounts to describing the
image of the coordinate map. Over the
tropical semifield, Corollary~\ref{cor:image} plays that role in geometric form: the
coordinates $w_c$ of a point of the image satisfy no further constraints
beyond the vanishing of one explicit tropical rational function built from
the basic values themselves.
\end{remark}

\begin{corollary}[subfamilies]\label{cor:subfamily}
Let $F$ be a subfamily of the basic family.
\begin{itemize}
\item[\textup{(i)}] If $F$ generates the invariant semifield, then
$\Phi_F := (b_c)_{c \in F}$ separates orbits.
\item[\textup{(ii)}] If $F$ contains all single-column values
$b_{i e_\alpha}$ ($i \in \{1, \dots, n\}$, $\alpha \in \{1, \dots, r\}$)
and $\Phi_F$ separates orbits, then $F$ generates.
\end{itemize}
Hence, for subfamilies containing the single-column values,
generating and separating are equivalent.
\end{corollary}

\begin{proof}
(i) If $F$ generates, each basic value $b_c$ is a tropical rational
expression in $F$, so $\Phi_F$ determines $\Phi$, which separates.
(ii) Each member of $F$ is a max filter in the sense of
Section~\ref{sec:general}, so the injective max filter bank
$\Phi_F$ is bi-Lipschitz \cite{BT}: there is $c_F > 0$ with
$\|\Phi_F(M) - \Phi_F(M')\|_2 \ge c_F\, d([M],[M'])$. The proof of
Theorem~\ref{th:gen} used only the single-column values (Lemma~\ref{lem:marg}),
the re-assembly lemma (Lemma~\ref{lem:sheet}), the discriminants with
their lower bound (Lemma~\ref{lem:disc}), and the Lipschitz bound
(Lemma~\ref{lem:lip}). Replacing $\Phi$ by $\Phi_F$ throughout --- the
discriminants restricted to $c \in F$, and
$c'_F := c_F/\sqrt{\#F}$ --- every step goes through verbatim, and
\eqref{eq:formula} expresses any invariant $f$ as a tropical rational
function of $F$.
\end{proof}

\begin{remark}[effectivity of $\kappa_f$]\label{rem:effective}
The constant $c_{n,r}$ of \eqref{eq:bilip} enters only through $\kappa_f$, and
any explicit lower bound on it makes the generating expressions effective
(its existence in \cite{K1} is qualitative). Alternatively,
Section~\ref{sec:algorithm} obtains $\kappa_f$ without any knowledge of $c_{n,r}$.
\end{remark}

\begin{remark}[expression size; the case $(n,r) = (2,2)$]\label{rem:size}
The formula \eqref{eq:formula} has $(n!)^{r-1}$ terms; no attempt at
optimality is made. For $(n,r) = (2,2)$ it reads as follows. The sorted
columns $x_1 \ge x_2$, $y_1 \ge y_2$ are the prefix differences
$x_1 = w_{10}$, $x_2 = w_{20} - w_{10}$, $y_1 = w_{01}$,
$y_2 = w_{02} - w_{01}$; the two couplings, \emph{aligned}
($\sigma = \mathrm{id}$) and \emph{crossed} ($\sigma = (1\,2)$),
re-assemble them into the matrices $\widehat M_{\mathrm{id}}$ and
$\widehat M_{(1\,2)}$ with rows $(x_1,y_1),\,(x_2,y_2)$ and
$(x_1,y_2),\,(x_2,y_1)$; and
\[
g = \min\bigl( f(\widehat M_{\mathrm{id}}) + \kappa_f D_{\mathrm{id}},\;\;
               f(\widehat M_{(1\,2)}) + \kappa_f D_{(1\,2)} \bigr),
\]
where $D_{\mathrm{id}} = |w_{11} - \max(x_1{+}y_2,\, x_2{+}y_1)|$ and
$D_{(1\,2)} = |w_{11} - (x_1{+}y_1)|$ (the other coordinates agree
automatically). As a worked instance, take
$f(M) = \max_i\, (m_{i1} + m_{i2})$, an invariant outside the basic
family: $f(\widehat M_{\mathrm{id}}) = x_1 + y_1$ and
$f(\widehat M_{(1\,2)}) = \max(x_1{+}y_2,\, x_2{+}y_1)$, so
\[
f \;=\; \min\bigl( x_1 + y_1 + \kappa_f D_{\mathrm{id}},\;\;
\max(x_1{+}y_2,\, x_2{+}y_1) + \kappa_f D_{(1\,2)} \bigr) \circ \Phi .
\]
Every $[M]$ equals $[\widehat M_{\mathrm{id}}]$ or $[\widehat M_{(1\,2)}]$
(or both, in case of ties). If $[M] = [\widehat M_{\mathrm{id}}]$ then
$w_{11} = \max(x_1{+}y_2,\, x_2{+}y_1)$, so the first term has
$D_{\mathrm{id}} = 0$ and value $f(M)$, while the penalty pushes the
second term to at least $f(M)$; if $[M] = [\widehat M_{(1\,2)}]$ then
$w_{11} = x_1 + y_1$ and the roles of the two terms are exchanged; in
either case the minimum returns $f(M)$. How small the
expressions can be made --- and, dually, which subfamilies of basic values still generate
--- connects to the minimal separating subfamilies
(Corollary~\ref{cor:subfamily}), which we do not pursue here.
\end{remark}

\subsection{An expression algorithm}\label{sec:algorithm}

Formula \eqref{eq:formula} is syntactic in $f$ except for the constant
$\kappa_f$. This yields the following algorithm. The input is a tropical rational
expression $f = P - Q$ in the variables $m_{i\alpha}$ that is invariant as a
function; the output is a tropical rational expression $g$ in the $K$
variables $w_c$ with $g \circ \Phi = f$.

\begin{enumerate}
\item \emph{Lipschitz bound.} Set
$L := \max_j \|a_j\|_1 + \max_k \|a'_k\|_1$, the sum of the
largest monomial slope $\ell^1$-norms of $P$ and $Q$; then $L \ge L_f$.
\item \emph{Coupling terms.} For each $\vec\sigma \in S_n^{\,r-1}$, build
$f_{\vec\sigma}$ by substituting
$m_{i\alpha} \leftarrow w_{\sigma_\alpha(i) e_\alpha} - w_{(\sigma_\alpha(i)-1) e_\alpha}$
(with $\sigma_1 = \mathrm{id}$) into $f$, and build $D_{\vec\sigma}$ as in
Lemma~\ref{lem:disc}, restricted to the coordinates $c$ with
at least two nonzero entries: single-column basic values of $M$ and of any
recoupling of $M$ coincide, so the restriction does not change
$D_{\vec\sigma}$ on the image of $\Phi$, and Lemma~\ref{lem:disc} holds
verbatim for it.
\item \emph{Constant.} If an explicit positive lower bound $c_0$
($\le c_{n,r}$) is
available, set $\kappa_f := \lceil L \sqrt{K} / c_0 \rceil$ and output
$g = \min_{\vec\sigma} ( f_{\vec\sigma} + \kappa_f D_{\vec\sigma} )$. Otherwise
run a doubling loop: for $\kappa_f = 1, 2, 4, \dots$ build $g_{\kappa_f}$ and test the
identity $g_{\kappa_f} \circ \Phi \equiv f$; output the first that passes.
\end{enumerate}

\begin{proposition}\label{prop:algorithm}
The algorithm is correct and terminates.
\end{proposition}

\begin{proof}
The identity test in step 3 is decidable. The expression
$h := g_{\kappa_f} \circ \Phi - f$ agrees with a single affine function on
each cell of the hyperplane arrangement generated by the breakpoints of
its constituent maxima. An interior point of a cell --- computable by
linear programming --- determines the maximizing term of each constituent
maximum on that cell, hence the affine function in closed form, and $h$
vanishes identically if and only if each of these finitely many affine
functions is the zero function, a check of the coefficients. By
Theorem~\ref{th:gen} the test passes as soon as $\kappa_f \ge L_f/c'_{n,r}$, so
the doubling loop halts, with $\kappa_f \le \max\{1,\, 2L_f/c'_{n,r}\}$.
\end{proof}

\begin{remark}[output size]\label{rem:prototype}
The output is the expression
$g = \min_{\vec\sigma} ( f_{\vec\sigma} + \kappa_f D_{\vec\sigma} )$ itself:
$(n!)^{r-1}$ coupling terms, each a copy of
$f$ with variables substituted plus $\kappa_f$ times a discriminant with at most
$K - nr$ absolute values ($\max$ commutes with positive scaling, so the
multiplication by $\kappa_f$ only scales exponents). At $(n,r) = (2,2)$ each
term carries a single absolute value, as in Remark~\ref{rem:size}.
\end{remark}

\section{Arbitrary permutation groups}\label{sec:general}

The proof of Theorem~\ref{th:gen} used nothing specific to row permutations beyond
two structural features, which we now isolate. Let $G \le S_N$ act on
$\R^N$ by coordinate permutations and $d$ be the quotient metric. Write
\[
\varepsilon_k(x) \;:=\; \max_{\#I = k} \sum_{i \in I} x_i,
\qquad k \in \{1, \dots, N\}, \qquad \varepsilon_0 := 0,
\]
for the tropical elementary symmetric polynomials. Equivalently,
$\varepsilon_k$ is the basic value
$b_{(k)}$ for $r = 1$: it is $S_N$-invariant, hence $G$-invariant, of
degree $k$, and $\varepsilon_k(x) - \varepsilon_{k-1}(x)$ is the $k$th
largest entry of $x$, so that $\varepsilon_1, \dots, \varepsilon_N$
recover the sorted vector $x^\downarrow$ (Lemma~\ref{lem:marg} for $r = 1$).

\begin{theorem}[axiomatic form]\label{th:axiomatic}
Let $\mathcal F = (f_1, \dots, f_{\#\mathcal F})$ be a finite tuple of
$G$-invariant tropical rational functions on $\R^N$ such that
\begin{itemize}
\item[\textup{(S)}] the map
$\Phi_{\mathcal F} : \R^N \to \R^{\#\mathcal F}$,
$x \mapsto (f_1(x), \dots, f_{\#\mathcal F}(x))$, separates $G$-orbits
and satisfies the lower Lipschitz bound
$\|\Phi_{\mathcal F}(x) - \Phi_{\mathcal F}(y)\|_2 \ge c_{\mathcal F}\, d([x],[y])$ for
some $c_{\mathcal F} > 0$ ($\Phi_{\mathcal F}$ is then in fact bi-Lipschitz: the upper bound is
automatic by Lemma~\ref{lem:lip}, each $f_i$ being tropical rational and
invariant);
\item[\textup{(T)}] there is a finite index set $J$ and componentwise
tropical rational maps $T_j : \R^{\#\mathcal F} \to \R^N$, $j \in J$, such that
every $x \in \R^N$
satisfies $[T_j(\Phi_{\mathcal F}(x))] = [x]$ for at least one $j$.
\end{itemize}
Then $\mathcal F$ generates the semifield of $G$-invariant tropical rational
functions: for every invariant $f$ and every sufficiently large
integer $\kappa_f$,
\[
f \;=\; \Bigl( \min_{j \in J} \bigl[ f \circ T_j + \kappa_f\, D_j \bigr] \Bigr) \circ \Phi_{\mathcal F},
\qquad
D_j(w) := \max_{1 \le i \le \#\mathcal F} \bigl| w_i - f_i(T_j(w)) \bigr| .
\]
\end{theorem}

\begin{proof}
As for Theorem~\ref{th:gen}, taking any integer $\kappa_f \ge L_f \sqrt{\#\mathcal F} / c_{\mathcal F}$
with $L_f$ an $\ell^\infty$-Lipschitz constant of $f$ (Lemma~\ref{lem:lip}). Each term is tropical rational. At $w = \Phi_{\mathcal F}(x)$,
\[
D_j(w) \;=\; \|\Phi_{\mathcal F}(x) - \Phi_{\mathcal F}(T_j(w))\|_\infty :
\]
this vanishes for any $j$ realizing \textup{(T)}, and in general is at least
$(c_{\mathcal F}/\sqrt{\#\mathcal F})\, d([x], [T_j(w)])$ by \textup{(S)}. The realized term
equals $f(x)$, and every term is at least
\[
f(x) - L_f\, d_\infty([x],[T_j(w)]) + \kappa_f\, (c_{\mathcal F}/\sqrt{\#\mathcal F})\, d([x],[T_j(w)])
\;\ge\; f(x),
\]
since $d_\infty \le d$ and $\kappa_f\, (c_{\mathcal F}/\sqrt{\#\mathcal F}) \ge L_f$.
\end{proof}

\begin{corollary}[quadratic degree bound for all permutation groups]\label{cor:quadratic}
For every permutation group $G \le S_N$, the semifield of $G$-invariant
tropical rational functions on $\R^N$ is generated by finitely many
invariant tropical polynomials of degree at most $\max\{N,\, \binom N2\}$.
\end{corollary}

\begin{proof}
Take $\mathcal F$ to be the separating, bi-Lipschitz family of
\cite[Theorem~1.3]{Derksen} --- of degree at most $\max\{N, \binom N2\}$ and
size $N + N!/|G|$, including
$\varepsilon_1, \dots, \varepsilon_N$. For \textup{(T)}: taking
consecutive differences of the $N$ coordinates carried by
$\varepsilon_1, \dots, \varepsilon_N$ defines a componentwise tropical
rational map $s : \R^{\#\mathcal F} \to \R^N$ (each component a difference of two
coordinates) with $s(\Phi_{\mathcal F}(x)) = x^\downarrow$; for $\pi \in S_N$ set
$T_\pi(w) := \pi \cdot s(w)$. Every $x$ equals
$\pi \cdot x^\downarrow$ for some $\pi$, so $[T_\pi(\Phi_{\mathcal F}(x))] = [x]$ for
that $\pi$. Theorem~\ref{th:axiomatic} applies with $J = S_N$; terms with $\pi$ in
the same right coset $G\pi$ coincide --- for $g \in G$ we have
$T_{g\pi}(w) = g \cdot T_\pi(w)$, and $f$ and the $f_i$ are $G$-invariant
--- so $N!/|G|$ terms suffice.
\end{proof}

\begin{remark}\label{rem:quadraticvsprimorial}
Corollary~\ref{cor:quadratic} replaces the degree bound
$N p_1 p_2 \cdots p_{|G|}$ of \cite[Theorem~1.2]{Derksen}, which depends
on the group order through the primorial factor, by a bound quadratic in
$N$ alone. The two bounds are incomparable: for groups of bounded order
the primorial bound is linear in $N$, while the quadratic bound is at most
the primorial one as soon as $p_1 \cdots p_{|G|} \ge N$ --- for instance
whenever $|G| \ge \log_2 N$ --- and for the row permutation action the
gain is superexponential. The row permutation case also shows
that group-specific families can lower the degree further (to $n$, with
$\binom{n+r}{r}-1$ generators); uniformly over all permutation groups,
however, the quadratic bound cannot be improved:
Theorem~\ref{th:alternating} and Corollary~\ref{cor:anlower} below show that every separating ---
in particular every generating --- family of $A_N$-invariant tropical
polynomials contains a member of degree at least $\binom N2$.
\end{remark}

The size $N + N!/|G|$ of the family used in Corollary~\ref{cor:quadratic} can also
be brought down to linear in $N$, at the price of the degree bound and of
explicitness. Here we use that $G$ acts on $\R^N$ by permutation matrices,
hence is a finite subgroup of the orthogonal group $\operatorname{O}(N)$;
this is the setting of \cite{CIMP} and \cite{BT}, both invoked below.
For $z \in \R^N$ write
$f_z(x) := \max_{g \in G}\, \langle x, g \cdot z \rangle$ for the
\emph{max filter} with template $z$; when $z \in \N^N$, $f_z$ is (the
evaluation of) a $G$-invariant tropical polynomial of degree $|z|_1$,
namely the $G$-trace $\bigoplus_{g \in G} x^{g \cdot z}$ of the tropical
monomial $x^z$. Note that
$\varepsilon_N = f_{(1,\dots,1)}$.

\begin{corollary}[linear-size separation by invariant tropical polynomials]\label{cor:count}
For every permutation group $G \le S_N$ there exist
$z_1, \dots, z_{2N} \in \N^N \setminus \{0\}$ such that the $2N + 1$
invariant tropical polynomials
\[
f_{z_1},\ f_{z_2},\ \dots,\ f_{z_{2N}},\ \varepsilon_N
\]
separate $G$-orbits and induce a bi-Lipschitz embedding of $\R^N / G$
into $\R^{2N+1}$. Every family of continuous
invariants that separates $G$-orbits has at least $N$ members.
\end{corollary}

\begin{proof}
The proof of \cite[Corollary~13]{CIMP} shows, for finite
$G \le \operatorname{O}(N)$ and $2N$ templates, that the set
$\mathcal Z \subseteq (\R^N)^{2N}$ of template tuples whose max filter
bank $(f_{z_1}, \dots, f_{z_{2N}})$ fails to separate $G$-orbits is
semialgebraic of dimension at most $2N^2 - 1$, less than the
dimension $2N^2$ of the ambient space $(\R^N)^{2N} \cong \R^{2N^2}$. Hence
almost every tuple separates: by \cite{Coste} the closure
$\overline{\mathcal Z}$ is
semialgebraic of the same dimension, still below $2N^2$, so it has empty
interior (a semialgebraic set containing a nonempty open set has full
dimension $2N^2$), and the separating tuples --- the complement of
$\overline{\mathcal Z}$ --- form an open dense set, which therefore
contains a rational tuple; scaling each template by a positive integer preserves
separation, since $f_{c z} = c\, f_z$ for $c > 0$, so we may take a
separating tuple $\tilde z_1, \dots, \tilde z_{2N} \in \Z^N$.
To make the $f_{z_i}$ tropical polynomials rather than merely rational, we
clear the negative entries: choose $u_i \in \N$ with
$z_i := \tilde z_i + u_i\, (1, \dots, 1) \in \N^N \setminus \{0\}$. Since
$\langle x, g \cdot (1,\dots,1) \rangle = \sum_j x_j$ for every
$g \in G$, we have $f_{z_i} = f_{\tilde z_i} + u_i\, \varepsilon_N$, so
the displayed family is the image of the separating family
$(f_{\tilde z_1}, \dots, f_{\tilde z_{2N}}, \varepsilon_N)$ under an
invertible affine change of coordinates on $\R^{2N+1}$; since the latter
separates $G$-orbits, so does the former. Being orbit-separating, the
displayed family is an injective max filter bank with nonzero windows
$z_1, \dots, z_{2N}, (1,\dots,1)$, so \cite[Corollary~1.5]{BT} yields the
bi-Lipschitz bound.

For the lower bound, the free locus of $G$ --- the set of points with
trivial stabilizer, that is, the complement of the finitely many proper
subspaces $\{x : g \cdot x = x\}$, $g \ne \mathrm{id}$ --- is open and
dense, and its image in $\R^N/G$ is a
topological $N$-manifold. A family of $m$ continuous invariants that
separates orbits induces a continuous injection of this manifold into
$\R^m$, which forces $m \ge N$ by invariance of domain.
\end{proof}

\begin{corollary}[linear-size generation]\label{cor:gencount}
For every permutation group $G \le S_N$, the semifield of $G$-invariant
tropical rational functions on $\R^N$ is generated by the $3N$ invariant
tropical polynomials
\[
f_{z_1},\ \dots,\ f_{z_{2N}},\
\varepsilon_1,\ \dots,\ \varepsilon_N
\]
with $z_1, \dots, z_{2N}$ as in Corollary~\ref{cor:count}. Every generating set has
at least $N$ members.
\end{corollary}

\begin{proof}
The family satisfies \textup{(S)} of Theorem~\ref{th:axiomatic}: it contains the
separating bi-Lipschitz family of Corollary~\ref{cor:count} as a sub-block, and
appending the coordinates $\varepsilon_1, \dots, \varepsilon_{N-1}$
preserves injectivity and, since adjoining coordinates only enlarges
$\|\cdot\|_2$, the lower Lipschitz bound with the same constant. Property \textup{(T)} holds by the
sorting argument in the proof of Corollary~\ref{cor:quadratic}, which uses only
$\varepsilon_1, \dots, \varepsilon_N$: the map $s$ built there from their
coordinates, together with $T_\pi(w) = \pi \cdot s(w)$, gives
$[T_\pi(\Phi_{\mathcal F}(x))] = [x]$ for some $\pi$, and $N!/|G|$ maps suffice.
Theorem~\ref{th:axiomatic} applies. For the lower bound, a generating set
separates orbits --- any tropical rational expression in the generators is
a function of their values, and the family of Corollary~\ref{cor:count} separates
orbits --- so Corollary~\ref{cor:count} gives $m \ge N$.
\end{proof}

\begin{remark}[count versus degree]\label{rem:countvsdegree}
Corollaries \ref{cor:count} and~\ref{cor:gencount} are optimal in the number of invariants up
to a factor of about $2$ and $3$ respectively, but control neither the
size of the template entries nor, consequently, the degrees $|z_i|_1$:
the genericity argument shows that suitable integer templates exist
without bounding their entries. Relative to \cite{CIMP}, where $2N$
random real templates are shown to separate,
the content of Corollary~\ref{cor:count} is that the templates can be taken in
$\N^N$, so that the separating family consists of invariant tropical
polynomials. The corollaries complement, rather than replace,
Corollary~\ref{cor:quadratic}, whose family has quadratic degree but size
$N + N!/|G|$; we know of no family of invariant tropical polynomials that
is simultaneously of linear size and of polynomially bounded degree, and
producing explicit templates is open. For the row permutation action of
$S_n$ on $\R^{n \times r}$ ($N = nr$), the basic values furnish a third
point of comparison: they are explicit and of degree at most $n$ --- below
even the quadratic degree of Corollary~\ref{cor:quadratic} --- but number
$\binom{n+r}{r}-1$, exponentially more than the $2nr + 1$ of
Corollary~\ref{cor:count} when $n$ and $r$ grow together (for $r = 1$ the two agree
up to a constant factor). Whether the row permutation action admits an
\emph{explicit} separating family that is at once of size polynomial in
$nr$ and of degree bounded by a polynomial in $n$ we leave open; the
minimal separating sets studied in \cite{LopatinReimers} are a point of
contact. The basic family also drives the explicit re-assembly of
Theorem~\ref{th:gen}.
\end{remark}

\medskip
The count lower bound of Corollary~\ref{cor:count} is complemented by a degree
lower bound, which shows that the quadratic bound of Corollary~\ref{cor:quadratic}
cannot be improved uniformly over all permutation groups: the alternating
group attains it. The proof rests on a normal form for invariant tropical
polynomials in terms of the traces $f_z$.

\begin{lemma}[trace normal form]\label{lem:tracenf}
Let $G \le S_N$ and let $f$ be a $G$-invariant tropical polynomial on
$\R^N$ of degree at most $d$. Then, as a function,
\[
f \;=\; \max_{j \in J}\, \bigl( a_j + f_{z_j} \bigr)
\]
for a finite set $J$, reals $a_j$, and templates $z_j \in \N^N$ with
$|z_j|_1 \le d$. Consequently, for $x, y \in \R^N$, the $G$-invariant
tropical polynomials of degree at most $d$ separate $[x]$ from $[y]$ if
and only if some trace $f_z$ with $z \in \N^N$, $|z|_1 \le d$, does.
\end{lemma}

\begin{proof}
Write $f(x) = \max_{j \in J} ( a_j + \langle x, z_j \rangle )$ with
$z_j \in \N^N$, $|z_j|_1 \le d$. Since permutation matrices are
orthogonal, $\langle g^{-1} \cdot x, z \rangle = \langle x, g \cdot z \rangle$,
so invariance gives
\[
f(x) \;=\; \max_{g \in G} f(g^{-1} \cdot x)
      \;=\; \max_{j \in J} \bigl( a_j + \max_{g \in G} \langle x, g \cdot z_j \rangle \bigr)
      \;=\; \max_{j \in J} \bigl( a_j + f_{z_j}(x) \bigr).
\]
For the second statement: by the normal form, equality of all such
traces at $x$ and $y$ forces equality of every $G$-invariant tropical
polynomial of degree at most $d$; conversely each such trace is itself
one.
\end{proof}

\begin{theorem}[rigidity of low-degree alternating invariants]\label{th:alternating}
Let $N \ge 2$. Every $A_N$-invariant tropical polynomial on $\R^N$ of
degree less than $\binom N2$ is $S_N$-invariant. The threshold
is exact: the $A_N$-trace $f_{z^\ast}$ with
$z^\ast = (0, 1, \dots, N-1)$ has degree $\binom N2$ and satisfies
$f_{z^\ast}(x) \ne f_{z^\ast}(\rho \cdot x)$ for every transposition
$\rho$ and every $x$ with pairwise distinct entries.
\end{theorem}

\begin{proof}
By Lemma~\ref{lem:tracenf} it suffices to prove that the $A_N$-trace $f_z$ is
$S_N$-invariant whenever $z \in \N^N$ has $|z|_1 < \binom N2$. Such a $z$
has two equal entries, since $N$ pairwise distinct nonnegative integers
sum to at least $0 + 1 + \cdots + (N-1) = N(N-1)/2 = \binom N2$. Let
$\rho \in S_N$ be a transposition exchanging two positions carrying equal
entries of $z$, so that $\rho \cdot z = z$; note
$h \cdot z = (h\rho) \cdot z$ for every $h \in S_N$. Let
$\sigma \in S_N$ be odd. Then
\[
f_z(\sigma^{-1} \cdot x)
 \;=\; \max_{g \in A_N} \langle x, (\sigma g) \cdot z \rangle
 \;=\; \max_{h \in \sigma A_N} \langle x, (h \rho) \cdot z \rangle
 \;=\; \max_{k \in \sigma A_N \rho} \langle x, k \cdot z \rangle
 \;=\; f_z(x),
\]
where the last equality holds because $A_N$ is normal and $\sigma\rho$ is
even, so $\sigma A_N \rho = A_N \sigma \rho = A_N$. As even $\sigma$
preserve $f_z$ by definition, $f_z$ is $S_N$-invariant.

For the threshold: let $x$ have pairwise distinct entries and let
$\rho$ be a transposition. Since $z^\ast$ also has pairwise distinct
entries, the maximum of $\langle x, g \cdot z^\ast \rangle$ over
$g \in S_N$ is attained at a unique $g^\ast \in S_N$: any $g$ that does
not pair the entries of $x$ and $z^\ast$ in matching sorted order admits
an improving swap. Now
$f_{z^\ast}(x) = \max_{g \in A_N} \langle x, g \cdot z^\ast \rangle$ is
the maximum over the even coset, while
$f_{z^\ast}(\rho \cdot x) = \max_{g \in A_N} \langle x, (\rho^{-1} g) \cdot z^\ast \rangle$
is the maximum over the odd coset; the coset containing $g^\ast$ attains
the global maximum and the other coset stays below it, so the
two values differ.
\end{proof}

\begin{corollary}[the quadratic degree bound is optimal]\label{cor:anlower}
Every family of $A_N$-invariant tropical polynomials that separates
orbits --- in particular, every generating family --- contains a
member of degree at least $\binom N2$. The bound is attained:
the $N + 1$ invariant tropical polynomials
$\varepsilon_1, \dots, \varepsilon_N, f_{z^\ast}$ with
$z^\ast = (0, 1, \dots, N-1)$ separate orbits, and their degrees
are at most $\max\{N, \binom N2\}$.
\end{corollary}

\begin{proof}
Let $x$ have pairwise distinct entries and let $\rho$ be a
transposition. The stabilizer of $x$ in $S_N$ is trivial and $\rho$ is
odd, so $\rho \cdot x \notin A_N \cdot x$. By Theorem~\ref{th:alternating},
every $A_N$-invariant tropical polynomial of degree less than
$\binom N2$ is $S_N$-invariant and hence takes equal values at $x$ and
$\rho \cdot x$; a family separating these two orbits must therefore
contain a member of degree at least $\binom N2$. Generating families
separate orbits, as in the proof of Corollary~\ref{cor:gencount}, so the same
bound applies to them.

For the displayed family: the values
$\varepsilon_1(x), \dots, \varepsilon_N(x)$ determine the multiset of
entries of $x$ (Lemma~\ref{lem:marg} for $r = 1$). If this multiset has a
repeated entry, the $S_N$-stabilizer of $x$ contains a transposition ---
an odd element --- so $A_N \cdot x = S_N \cdot x$ and the multiset alone
determines the orbit. If the entries are pairwise distinct, exactly two
orbits share the multiset, namely $[x]$ and $[\rho \cdot x]$ for
any transposition $\rho$, and $f_{z^\ast}$ takes distinct values on them
by Theorem~\ref{th:alternating}. In either case the values of the family
determine the orbit.
\end{proof}

\begin{remark}[scope of the obstruction]\label{rem:anscope}
(i) The proof of Theorem~\ref{th:alternating} shows, more generally: if
$\rho \notin G$ is a transposition normalizing $G$, then every $G$-trace
whose template has equal entries in the two positions exchanged by
$\rho$ ($\rho \cdot z = z$) is invariant under the group
$\langle G, \rho \rangle$ generated by $G$ and $\rho$, hence blind to the coset move
$x \mapsto \rho \cdot x$. What makes $A_N$ extremal is that it has index
two and every transposition lies in the unique nontrivial coset, so a
template with two equal entries, wherever they lie, is blind to the
\emph{same} pair of orbits $([x],\, [\rho \cdot x])$. Thresholds far below $\binom N2$ occur: for $G = S_N$,
already $\varepsilon_1, \dots, \varepsilon_N$ separate orbits in
degree $N$ (Lemma~\ref{lem:marg} for $r = 1$). As a concrete instance
of Theorem~\ref{th:alternating}, the integer pair $x = (1, 2, 3, 4)$,
$y = (2, 1, 3, 4)$ lies in two distinct $A_4$-orbits, every $A_4$-trace
of degree at most $5$ takes equal values at $x$ and $y$, and each of the
$24$ traces with template a permutation of $(0,1,2,3)$ separates them,
as in the proof of the theorem. Low-degree
non-separation is not confined to $A_N$: for the Frobenius group
$F_{20} = \langle (1\,2\,3\,4\,5),\ (2\,3\,5\,4) \rangle \le S_5$, the
integer pair $x = (1, 3, 2, 2, 0)$, $y = (1, 2, 2, 3, 0)$ lies in two
distinct $F_{20}$-orbits: an element mapping $x$ to $y$ would fix the
positions of the values $1$ and $0$, and a nonidentity element of a
Frobenius group fixes at most one point. No $F_{20}$-trace of degree
at most $5 = N$ separates the pair --- a finite check over the $16$
classes, up to the
$F_{20}$-action, of primitive templates of degree at most $5$, scaling
covering the rest --- while the
degree-$6$ trace with template $(0,1,2,2,1)$ takes the values $13$ at
$x$ and $12$ at $y$; no transposition normalizes
$F_{20}$, so this pair is not
an instance of the mechanism of Theorem~\ref{th:alternating}.
(ii) Negative exponents do not change the quadratic order. Call a finite
maximum of affine functions with slopes in $\Z^N$ a tropical
\emph{Laurent} polynomial, of degree the largest
$|z|_1 = \sum_{i \in \{1, \dots, N\}} |z_i|$ over its slopes $z$;
Lemma~\ref{lem:tracenf} and its proof apply verbatim. A vector $z \in \Z^N$
with pairwise distinct entries has $|z|_1 \ge \lfloor N^2/4 \rfloor$, so
the argument of Theorem~\ref{th:alternating} shows that every $A_N$-invariant
tropical Laurent polynomial of degree less than $\lfloor N^2/4 \rfloor$
is $S_N$-invariant: negative exponents improve the constant from
$\tfrac12$ to $\tfrac14$ but not the exponent of $N$.
(iii) Corollary~\ref{cor:anlower} strengthens Remark~\ref{rem:countvsdegree} for
$G = A_N$: a separating family of linear size and polynomially bounded
degree, if it exists, must contain a member of degree at least
$\binom N2$.
\end{remark}

\subsection*{Declaration of competing interest}

The author declares no competing interests.

\subsection*{Declaration of generative AI and AI-assisted technologies in
the manuscript preparation process}

During the preparation of this work the author used Claude (Anthropic) in
order to refine the exposition and to check the arguments. After using
this tool, the author reviewed and edited the content as needed and takes
full responsibility for the content of the published article.

\subsection*{Data availability}

No data was used for the research described in the article.

\end{document}